\documentclass[a4paper,twoside]{amsart}
\usepackage{xypic,amscd,amssymb,latexsym,srcltx}
\usepackage{amsthm}
\def\@setcopyright{}
\makeatother

\newtheorem{lemma}{Lemma}[section]
\newtheorem{proposition}[lemma]{Proposition}

\newtheorem{theorem}[lemma]{Theorem}
\newtheorem{remark}[lemma]{Remark}

\newtheorem*{acknowledgments*}{ACKNOWLEDGMENTS}
\newtheorem{definition}[lemma]{Definition}
\newtheorem{example}[lemma]{Example}

\newtheorem*{conj*}{Conjecture}
\newtheorem*{thm*}{Main Theorem}
\newtheorem*{remark*}{Remark}
\newtheorem*{lemma*}{Lemma}
\newtheorem*{example*}{[Example]}

\begin{document}
\begin{center}
{\Large \bf A bound for Hall's criterion for nilpotence in semi-abelian categories
}\\

\bigskip
\end{center}

\begin{center}
Heguo Liu$^{1}$, Xingzhong Xu$^{1}$, Jiping Zhang$^{2}$

\end{center}

\footnotetext {$^{*}$~~Date: May/21/2019.
}
\footnotetext {
1. Department of Mathematics, Hubei University, Wuhan, 430062, China

2. School of Mathematical Sciences, Peking University, Beijing, 100871, China

Heguo Liu's E-mail: ghliu@hubu.edu.cn

Xingzhong Xu's E-mail: xuxingzhong407@hubu.edu.cn; xuxingzhong407@126.com

Jiping Zhang's E-mail:  jzhang@pku.edu.cn

Supported by National 973 Project (2011CB808003) and NSFC grant (11371124, 11501183).
}

\title{}
\def\abstractname{\textbf{Abstract}}

\begin{abstract}\addcontentsline{toc}{section}{\bf{English Abstract}}
In this paper, we focus on Hall's criterion for nilpotence in semi-abelian categories,
and we improve the bound of the main theorem of \cite[Theorem 3.4]{Gr}
(see Main Theorem). And this bound is best possible.
\hfil\break

\hfil\break \textbf{2000 Mathematics Subject
Classification:} \  $\cdot$ \  $\cdot$ \

\end{abstract}

\maketitle

\section{\bf Introduction}

In \cite{Gr}, Gray has proved a wide generalization of P. Hall's theorem about nilpotent groups:
a group $G$ is nilpotent if it has a normal subgroup $N$ such that $G/[N,N]$ and $N$ is nilpotent.
Gray's generalization is in a semi-abelian category \cite{JMT} which satisfies some properties\cite[Section 3]{Gr}.
Moreover, Gray's main theorem gives a  bound  of  the nilpotency class about the similar objects in  
algebraically coherent semi-abelian category (see \cite[Theorem 3.4]{Gr}). 
In this note, we improve the bound as follows.

\begin{thm*} Let $\mathbb{C}$ be an algebraically coherent semi-abelian category and let
$p: E\to B$ be an extension of a nilpotent object $B$ in $\mathbb{C}$. If
the kernel of $p$ is contained in the Huq commutator $[N, N]_N$ of a nilpotent normal
subobject $N$ of $E$, and if $N$ is of nilpotency class $c$ and $B$ is of nilpotency class $d$,
then $E$ is of nilpotency class at most $cd+(c-1)(d-1)$.
\end{thm*}

Here, the definition of algebraically coherent semi-abelian category can be found in \cite{Gr}.
Examples of algebraically coherent semi-abelian categories include the categories of groups, rings,
Lie algebra over a commutative ring, and others categories in \cite{Or}.
And the bound in the categories of groups is found by \cite[Theorem 1.]{S}.

$Structure~ of ~ the~ paper:$
After recalling the basic definitions and properties of commutator semi-lattices in Section 2, and
we introduce semi-abelian categories and commutators in Section 3. In Section 4,
we prove Main Theorem.

\section{\bf Jacobi commutator semi-lattices}

In this section we collect some known results about commutator semi-lattices.  For the background theory of commutator semi-lattices, we refer to \cite{Gr}.

First, let us begin with semi-lattices.
\begin{definition} A $\mathrm{semi}$-$\mathrm{lattice}$ is a
triple $(X, \leq, \vee)$ where $(X, \leq)$ is a poset, and $\vee$ is a binary
operation on $X$ satisfying:

(a) for each $a\in X$, $a\vee a=a$;

(c) for each $a, b \in X, a\vee b = b\vee a$;

(d) for each $a, b, c \in X, (a\vee b)\vee c=a\vee (b\vee c)$.

Moreover, a semi-lattice $(X, \leq, \vee)$ is called $\mathrm{join}$ if
$$a\leq b\Longleftrightarrow a\vee b = b$$
for each $a, b\in X$.
\end{definition}

\begin{definition} A $\mathrm{commutator~semi}$-$\mathrm{lattice}$ is a
triple $(X, \leq, \cdot)$ where $X$ is a set, $\leq$ is a binary relation on $X$, and $\cdot$ is a binary
operation on $X$ satisfying:

(a) $(X, \leq)$ is a join semi-lattice;

(b) the operation $\cdot$ is commutative;

(c) for each $a, b \in X, a\cdot b \leq b$;

(d) for each $a, b, c \in X, a \cdot (b \vee c) = (a \cdot b) \vee (a \cdot c)$.

\end{definition}

\begin{remark} Let $(X, \leq, \cdot)$ be a commutator semi-lattice, and $x\in X$.
Then the map $x\cdot-: X\to X$ defined by $y\mapsto x\cdot y$ is order preserving.
\end{remark}

\begin{proof} Let $y\leq z\in X$, then $y\vee z=z$.
Since $(x\cdot y)\vee (x\cdot z)=x\cdot (y\vee z)$ by above definition (d),
we have $(x\cdot y)\vee (x\cdot z)=x\cdot (y\vee z)=x\cdot z$.
Hence, $x\cdot y\leq x\cdot z$.
\end{proof}

\begin{definition} A commutator semi-lattice $(X, \leq, \cdot)$ is a $\mathrm{Jacobi}$ commutator
semi-lattice if

(a) for each  $a, b, c \in X,  a \cdot (b \cdot c) \leq ((a \cdot b) \cdot c) \vee (b \cdot (a \cdot c))$;
\end{definition}

\begin{example} Let $G$ be a group, and let $X$ be the set of all normal subgroups of $G$.
For each $M, N\in X$, we can define that $N\cdot M=[M,N]$ and $N\vee M=NM$.
It is easy to see that $(X, \leq, \cdot)$ is a Jacobi commutator semi-lattice.
\end{example}

\begin{definition} A derivation of a commutator semi-lattice $(X, \leq, \cdot)$ is a map
$f: X\to X$ which preserves joins and satisfies:

(a) for each $a, b \in X, f(a \cdot b) \leq (f(a) \cdot b) \vee (a \cdot f(b))$.

A derivation $f$ of a commutator semi-lattice $(X, \leq, \cdot)$ is an inner derivation if
there exists $x$ in $X$ such that $f = x \cdot -$, that is, for each $a$ in $X, f(a) = x \cdot a$.
\end{definition}

\begin{remark} Let $f$ be a derivation of commutator semi-lattice $(X, \leq, \cdot)$.
For each $a, b$ in $X$ and $a\leq b$, then $f(a)\leq f(b)$.
\end{remark}

\begin{proof} Since $a\leq b$, we have $a\vee b=b$. Also $f$ is a derivation, thus $f$ preserves
joins. Hence, $f(b)=f(a\vee b)= f(a)\vee f(b)$. So $f(a)\leq f(b)$.
\end{proof}

\begin{proposition}Let $g$ be the inner derivation of a Jacobi commutator semi-lattice $(X, \leq, \cdot)$,
let $x$ be an elements of $X$, and let $g$ be
defined for each $s$ in $X$ by $g(s)=x\cdot s$. Then
$$g^i(x)\cdot g^j(x)\leq g^{i+j+1}(x)$$
for each each non-negative integers $i$ and $j$.
\end{proposition}

\begin{proof} The proof is by induction on $j$. For $j=0$,
 we can see $g^i(x)\cdot g^0(x)=g^i(x)\cdot x= x \cdot g^i(x)= g^{i+1}(x)$.

We can see that
\begin{eqnarray*}
& ~&g^{i}(x)\cdot g^{j}(x)\\
&=& \underbrace{(x\cdot(x\cdot\cdots\cdot(x}_i\cdot x)))\cdot\underbrace{(x\cdot(x\cdot\cdots\cdot(x}_j\cdot x)))\\
&\leq & (\underbrace{(x\cdot(x\cdot\cdots\cdot(x}_i\cdot x)))\cdot x)\cdot\underbrace{(x\cdot(x\cdot\cdots\cdot(x}_{j-1}\cdot x)))\\
&~&\bigvee x\cdot(\underbrace{(x\cdot(x\cdot\cdots\cdot(x}_i\cdot x)))\cdot\underbrace{(x\cdot(x\cdot\cdots\cdot(x}_{j-1}\cdot x))))\\
&\leq& \underbrace{(x\cdot(x\cdot\cdots\cdot(x}_{i+1}\cdot x)))\cdot\underbrace{(x\cdot(x\cdot\cdots\cdot(x}_{j-1}\cdot x)))\\
&~&\bigvee x\cdot g^{i+j}(x)\\
&=& \underbrace{(x\cdot(x\cdot\cdots\cdot(x}_{i+1}\cdot x)))\cdot\underbrace{(x\cdot(x\cdot\cdots\cdot(x}_{j-1}\cdot x)))\bigvee g^{i+j+1}(x)\\
&=& g^{i+1}(x)\cdot g^{j-1}(x)\bigvee g^{i+j+1}(x)\\
&\leq & g^{i+j+1}(x)\bigvee g^{i+j+1}(x)\\
&=& g^{i+j+1}(x).
\end{eqnarray*}
Hence, we get the proof.
\end{proof}

\begin{proposition} Let $f$ be a derivation of commutator semi-lattice $(X, \leq, \cdot)$.
For each $a, b$ in $X$ and for each non-negative integer $n$, we have
$$f^n(a\cdot b)\leq \bigvee_{i=0}^n f^i(a)\cdot f^{n-i}(b).$$
\end{proposition}

\begin{proof} The proof is by induction on $n$. For $n=0$,
it follows by $f^0(a\cdot b)=a\cdot b\leq a\cdot b=f^0(a)\cdot f^0(b)$.
For $n=1$, we can see that
$$f(a\cdot b)\leq (f(a)\cdot b)\vee (a\cdot f(b))$$
for each $a, b\in X$ by the definition of derivation.

Now, we can assume that
the proposition hold for $n-1$. That means
$$f^{n-1}(a\cdot b)\leq \bigvee_{i=0}^{n-1} f^i(a)\cdot f^{n-1-i}(b).$$
By Remark 2.5 and the definition of derivation, we have
$$ f^n(a\cdot b)=f(f^{n-1}(a\cdot b))\leq f(\bigvee_{i=0}^{n-1} f^i(a)\cdot f^{n-1-i}(b))\leq \bigvee_{i=0}^{n-1} f(f^i(a)\cdot f^{n-1-i}(b)).$$
So, we can see that
\begin{eqnarray*}
f^n(a\cdot b)
&\leq & \bigvee_{i=0}^{n-1} f(f^i(a)\cdot f^{n-1-i}(b))\\
&\leq & \bigvee_{i=0}^{n-1} ((f^{i+1}(a)\cdot f^{n-1-i}(b))\vee (f^{i}(a)\cdot f^{n-i}(b)))\\
&\leq&((f^{1}(a)\cdot f^{n-1}(b))\vee (a\cdot f^{n}(b)))\\
&~& \bigvee ((f^{2}(a)\cdot f^{n-2}(b))\vee (f^1(a)\cdot f^{n-1}(b)))\\
&~& \bigvee ((f^{3}(a)\cdot f^{n-3}(b))\vee (f^2(a)\cdot f^{n-2}(b)))\\
&~&  ~~~~~~~~~\vdots \\
&~& \bigvee ((f^{n}(a)\cdot b)\vee (f^{n-1}(a)\cdot f(b)))\\
&=& \bigvee_{i=0}^n f^i(a)\cdot f^{n-i}(b).
\end{eqnarray*}
Hence, we get the proof.
\end{proof}

\begin{proposition} Let $f$ be a derivation of a Jacobi commutator semi-lattice $(X, \leq, \cdot)$ bounded
above by $1_X$, let $x$ and $y$ be an elements of $X$, and let $g$ be the inner derivation of $(X, \leq, \cdot)$
defined for each $s$ in $X$ by $g(s)=x\cdot s$. If $x\leq y$ and for some positive integer $m$, $f^m(y)\leq g(x)$,
then for each positive integer $k$,
$$f^{t_k}(y)\leq g^k(x)$$
where $t_k=km+(k-1)(m-1).$
\end{proposition}

\begin{proof} The proof is by induction on $k$.

{\bf Step 1.} For $k=1$, we can see $t_1=m$. So we can see the condition $f^{t_k}(y)\leq g^k(x)$
makes this case holds.

{\bf Step 2.} If $k> 1$, then for $r\leq k-1$, we can assume that
the proposition hold when $r\leq k-1$. That means that
$$f^{t_r}(y)\leq g^r(x)$$
for each $1\leq r\leq k-1$.

{\bf Step 3.} For $k$, we can see that
$f^{t_k}(y)\leq f^{t_k-m}(f^m(y))\leq f^{t_k-m}(g(x))$ by Remark 2.7.
And by Proposition 2.9, we have
\begin{eqnarray*}f^{t_k}(y)
&\leq & f^{t_k-m}(f^m(y))\leq f^{t_k-m}(g(x))=f^{t_k-m}(x\cdot x)\\
&\leq & \bigvee_{i=0}^{t_k-m} f^i(x)\cdot f^{t_k-m-i}(x).
\end{eqnarray*}

Now, we will consider $ f^i(x)\cdot f^{t_k-m-i}(x)$ for each $i$.
For each $i$, there exists $1\leq j\leq k$ such that
$$2(j-1)m-m-(j-1)+1\leq i< 2jm -m -j +1.$$

{\bf For $f^i(x)$,} we can see that
$$f^i(x)\leq f^i(y)\leq f^{2(j-1)m-m-(j-1)+1}(y).$$
Here, $2(j-1)m-m-(j-1)+1=t_{j-1}$. But $j-1\leq k-1$, thus
$$f^{2(j-1)m-m-(j-1)+1}(y)=f^{t_{j-1}}(y)\leq g^{j-1}(x).$$

{\bf For $ f^{t_k-m-i}(x)$,} we can see that
$$f^{t_k-m-i}(x)\leq f^{t_k-m-i}(y).$$
But
\begin{eqnarray*}t_k-m-i
&=& 2km-k-2m+1-i-2jm- j + 2jm +j\\
&=& 2(k-j)m-(k-j)-m+1+jm-m-j-i\\
&=& t_{k-j}+jm-m-j-i.
\end{eqnarray*}
Since $i< 2jm -m -j +1$, we have $jm-m-j-i\geq 0$.
Hence, we have
\begin{eqnarray*}f^{t_k-m-i}(x)\leq f^{t_k-m-i}(y)
&=& f^{t_{k-j}+jm-m-j-i}(y)\\
&=& f^{t_{k-j}}(f^{jm-m-j-i}(y))\\
&\leq & f^{t_{k-j}}(y)\\
&\leq & g^{k-j}(x).
\end{eqnarray*}

Hence, we have
$$f^i(x)\cdot f^{t_k-m-i}(x)\leq g^{j-1}(x)\cdot g^{k-j}(x)\leq g^{k-j+j-1+1}(x)=g^k(x).$$
Therefore, we can see
\begin{eqnarray*}f^{t_k}(y)
&\leq & f^{t_k-m}(f^m(y))\leq f^{t_k-m}(g(x))=f^{t_k-m}(x\cdot x)\\
&\leq & \bigvee_{i=0}^{t_k-m} f^i(x)\cdot f^{t_k-m-i}(x)\\
&\leq & \bigvee_{i=0}^{t_k-m} g^{j-1}(x)\cdot g^{k-j}(x)\\
&\leq & \bigvee_{i=0}^{t_k-m} g^k(x) ~--------\mathrm{by ~Proposition~ 2.8}\\
&=& g^k(x)
\end{eqnarray*}
 and we prove this proposition.
\end{proof}
The above proof follows \cite[Theorem 1]{S}.

\section{\bf Semi-abelian categories}

In this section we collect some known results about semi-abelian categories.  For the background theory of semi-abelian categories,
we refer to \cite{JMT}.

%
%
%
%
%
%
%
%
%
%
%
%
%
%

We introduce the Huq commutator as follows, and we use the notations in \cite{Gr}.

\begin{definition} Let $\mathcal{C}$ be a semi-abelian category.
Denote by $0$ a zero object in $\mathcal(C)$, and denote by $0$ each zero morphism, that is,
a morphism which factors through a zero object. For each $A, B\in \mathrm{Ob}(\mathcal{C})$,
we have a product $A\times B\in \mathrm{Ob}(\mathcal{C})$. By the definition of product,
we can write  $\langle 1, 0 \rangle: A\to A\times B$ and $\langle 0, 1 \rangle: B\to A\times B$
for the unique morphisms with $\pi_1\langle 1, 0 \rangle:=1_A, \pi_2\langle 1, 0 \rangle:=0,\pi_1\langle 0, 1\rangle:=0$ and
$\pi_2\langle 0, 1\rangle:=1_B$.
A pair of morphisms $f: A\to C$ and $g: B\to C$ {\bf commute}, if there is a morphism $\varphi: A\times B\to C$
making the diagram commute.
$$\CD
  A @> \langle 1, 0 \rangle >> A\times B @> \langle 0,1 \rangle >> B \\
  @V f VV @V \varphi VV @V g VV  \\
  C @>=>> C @>=>> C
\endCD
$$

More generally the {\bf Huq commutator} of  $f: A\to C$ and $g: B\to C$ is defined
to be the smallest normal subobject $N$ of $C$ such that $qf$ and $qg$ commute, where
$q: C\to C/N$ is the cokernel of the associated normal monomorphism $N\to C$.

\end{definition}

\begin{theorem} Let $\mathcal{C}$ be a semi-abelian category. Let  $f: A\to C$ and $g: B\to C$
be morphisms in $\mathcal{C}$, then there exists the Huq commutator of $f$ and $g$.
\end{theorem}

\begin{proof} See \cite{B} or \cite{Gr}.
\end{proof}

We recall the definition of nilpotent for object in a semi-abelian category $\mathcal{C}$ as follows.
\begin{definition} For subobjects $S$ and $T$ of an object $C$
in $\mathcal{C}$,  we will write $[S, T]_C$ for the Huq commutator of the associated monomorphisms
$S\to C$ and $T\to C$. Recall also that C is nilpotent if there exists a
non-negative integer $n$ such that $\gamma^n_C(C)=0$,
 where  $\gamma_C$ is the map sending $S$ in $\mathbf{Sub}(C)$
to $[C, S]_C$ in $\mathbf{Sub}(C)$. The least such $n$ is the nilpotency class of $C$.

\end{definition}

\section{\bf Proof of the main theorem}
In this section, we give a proof of the main theorem.

\begin{theorem} Let $\mathbb{C}$ be an algebraically coherent semi-abelian category and let
$p: E\to B$ be an extension of a nilpotent object $B$ in $\mathbb{C}$. If
the kernel of $p$ is contained in the Huq commutator $[N, N]_N$ of a nilpotent normal
subobject $N$ of $E$, and if $N$ is of nilpotency class $c$ and $B$ is of nilpotency class $d$,
then $E$ is of nilpotency class at most $cd+(c-1)(d-1)$.
\end{theorem}

\begin{proof} Let $X:= \mathbf{NSub}(E)$(Here, $\mathbf{NSub}(E)$ means all normal subobjects of $E$) and let 
$f: X\to X$ and $g: X\to X$ be the maps defined by 
$f(K)=[E, K]_E$ and $g(K)=[N, K]_E$. Using the proof of \cite[Theorem 3.4]{Gr},
we find $f^d(E)\leq g(N)$. So by Proposition 2.10, we get
$$f^{cd+(c-1)(d-1)}(E)\leq g^c(N).$$
So, we get the proof because $N$ is of nilpotency class $c$.
\end{proof}

\begin{example}\cite[Section 5, Example]{S} For every pair $c, d$ of positive integers there is a group
$G$ of class $cd+(c-1)(d-1)$ which has a normal subgroup $N$ of class $c$ such that $G/[N,N]$ is of
class $d$.
\end{example}

\textbf{ACKNOWLEDGMENTS}\hfil\break

The authors are grateful to the website address: ncatlab.org/nlab/show/HomePage for its guidance.

\end{document}